\documentclass[12pt,reqno]{amsart}
\usepackage{amsmath, amsfonts,epsfig, amssymb,amsbsy,eucal,mathrsfs,float,mathtools}
\usepackage[normalem]{ulem}

\usepackage{graphicx,xcolor}
\usepackage{dynkin-diagrams}

\usepackage{enumitem}
\usepackage{tikz}

\usepackage[all]{xy}

\usepackage{tikz-cd}
\usepackage{tikz}

\usepackage{hyperref}
\hypersetup{
colorlinks = true,
urlcolor = blue,
citecolor = .
}

\usepackage{fullpage}

\usepackage{todonotes}

\numberwithin{equation}{section}
\theoremstyle{plain}

\newtheorem{theorem}{Theorem}[section]
\newtheorem{lemma}[theorem]{Lemma}
\newtheorem{proposition}[theorem]{Proposition}

\newtheorem{conjecture}[theorem]{Conjecture}

\theoremstyle{definition}

\newtheorem{definition}[theorem]{Definition}

\theoremstyle{remark}

\newtheorem{remark}[theorem]{Remark}

\DeclareMathOperator{\G}{Gr}
\DeclareMathOperator{\IG}{IGr}

\DeclareMathOperator{\Fl}{Fl}

\DeclareMathOperator\HH{H}

\DeclareMathOperator{\Ext}{Ext}
\DeclareMathOperator{\Hom}{Hom}

\DeclareMathOperator{\Cone}{Cone}

\DeclareMathOperator{\Spin}{Spin}
\DeclareMathOperator\Sym{S}

\newcommand{\bL}{{\mathbb L}}

\newcommand{\bS}{{\mathbb S}}

\newcommand{\bP}{{\mathbb P}}

\newcommand{\bZ}{{\mathbb Z}}

\newcommand{\rA}{\mathrm{A}}
\newcommand{\rB}{\mathrm{B}}
\newcommand{\rC}{\mathrm{C}}
\newcommand{\rD}{\mathrm{D}}
\newcommand{\rE}{\mathrm{E}}
\newcommand{\rF}{\mathrm{F}}
\newcommand{\rG}{\mathrm{G}}

\newcommand{\rK}{\mathrm{K}}
\newcommand{\rL}{\mathrm{L}}
\newcommand{\rP}{\mathrm{P}}
\newcommand{\rS}{\mathrm{S}}

\newcommand{\cA}{\mathcal{A}}
\newcommand{\cB}{\mathcal{B}}
\newcommand{\cC}{\mathcal{C}}
\newcommand{\cD}{\mathcal{D}}
\newcommand{\cE}{\mathcal{E}}
\newcommand{\cF}{\mathcal{F}}
\newcommand{\cG}{\mathcal{G}}
\newcommand{\cH}{\mathcal{H}}

\newcommand{\cO}{\mathcal{O}}
\newcommand{\cR}{\mathcal{R}}

\newcommand{\cU}{\mathcal{U}}

\newcommand{\Db}{{\mathbf D^{\mathrm{b}}}}

\newcommand{\bq}{\mathbf{q}}
\newcommand{\bh}{\mathbf{h}}

\DeclareMathOperator{\Ca}{\mathbf{C}}

\newcommand{\brG}{\bar{\rG}}
\newcommand{\brL}{\bar{\rL}}
\newcommand{\brP}{\bar{\rP}}

\newcommand{\bcA}{\bar{\cA}}

\newcommand{\bcC}{\bar{\cC}}
\newcommand{\bcE}{\bar{\cE}}

\newcommand{\bcR}{\bar{\cR}}
\newcommand{\bcU}{\bar{\cU}}

\newcommand{\bom}{\bar\omega}

\begin{document}

\title[Derived categories of $\mathrm{E}_6/\mathrm{P}_1$ and $\mathrm{F}_4/\mathrm{P}_4$]{Derived categories of the Cayley plane\\[1ex] and the coadjoint Grassmannian of type~F}


\author{Pieter Belmans}
\address{\parbox{0.95\textwidth}{
Institut f\"ur Mathematik,
Universit\"at Bonn,
Endenicher Allee~60,
53115 Bonn,
Germany
\\[3pt]
Departement Wiskunde
Universiteit Antwerpen,
Middelheimlaan~1,
2020 Antwerpen,
Belgium
\smallskip
}}
\email{pbelmans@math.uni-bonn.de}

\author{Alexander Kuznetsov}
\address{
\parbox{0.95\textwidth}{
Algebraic Geometry Section, Steklov Mathematical Institute of Russian Academy of Sciences,
8 Gubkin str., Moscow 119991 Russia\\[3pt]
Laboratory of Algebraic Geometry, NRU HSE, Moscow, Russia\
\smallskip
}
}
\email{akuznet@mi-ras.ru}

\author{Maxim Smirnov}
\address{
\parbox{0.95\textwidth}{
Universit\"at Augsburg,
Institut f\"ur Mathematik,
Universit\"atsstr.~14,
86159 Augsburg,
Germany
\smallskip
}}
\email{maxim.smirnov@math.uni-augsburg.de}

\thanks{P.B. was partially supported by a postdoctoral fellowship from the Research Foundation-Flanders (FWO).
A.K. was partially supported by the HSE University Basic Research Program and the Russian Academic Excellence Project ``5-100''. M.S. was partially supported by the Deutsche Forschungsgemeinschaft (DFG, German Research Foundation) --- Projektnummer 448537907.}

\begin{abstract}
For the derived category of the Cayley plane, which is the cominuscule Grassmannian of Dynkin type~$\mathrm{E}_6$,
a full Lefschetz exceptional collection was constructed by Faenzi and Manivel.
A general hyperplane section of the Cayley plane is the coadjoint Grassmannian of Dynkin type~$\mathrm{F}_4$.
We show that the restriction of the Faenzi--Manivel collection
to such a hyperplane section gives a full Lefschetz exceptional collection,
providing the first example of a full exceptional collection on a homogeneous variety of Dynkin type~$\mathrm{F}$.

We also describe the residual categories of these Lefschetz collections,
confirming conjectures of the second and third named author for the Cayley plane and its hyperplane section.
The latter description is based on a general result of independent interest,
relating residual categories of a variety and its hyperplane section.
\end{abstract}

\maketitle


\section{Introduction}

Rational homogeneous varieties form an important and well-studied class of varieties,
whose geometry can be understood using the representation theory of algebraic groups.
The derived category of such a variety is conjectured to have a full exceptional collection,
and in many instances a construction is known.
For an overview of the state-of-the art (from a few years ago), one is referred to~\cite[\S1.1]{KuPo}.
The main advances since then are \cite{Fo19,Gu18,KSnew},
where exceptional collections were constructed on some homogeneous varieties of symplectic and orthogonal groups.

The part of the story corresponding to classical groups (Dynkin types $\rA$, $\rB$, $\rC$, and~$\rD$)
and the exceptional group of type~$\rG_2$
is relatively well studied, see~\cite{KuPo} and~\cite[\S6.4]{K06}.
For exceptional groups of Dynkin types~$\rE$ and $\rF$ on the contrary, very little is known.
The only homogeneous variety for these types that was known to have a full exceptional collection
is the cominuscule $\mathrm{E}_6$-Grassmannian (also called the \emph{Cayley plane})~\cite{FM}.

The main result of this paper is the construction of a full exceptional collection
for the coadjoint $\mathrm{F}_4$-Grassmannian (see~\eqref{eq:dby-intro} below and Theorem~\ref{theorem:f4} in the body of the paper for a more precise statement).

\begin{theorem}
\label{theorem:f4-intro}
Let $Y$ be the coadjoint Grassmannian of Dynkin type~$\rF_4$.
The derived category~$\Db(Y)$ has a full exceptional collection of length~$24$, consisting of equivariant vector bundles.
\end{theorem}

Note that the big quantum cohomology ring of~$Y$ is generically semisimple~\cite[Theorem~5.3.2]{Pe},
so Theorem~\ref{theorem:f4-intro} provides yet another instance where the first part of Dubrovin's conjecture is known to hold.

Our approach is based on the \emph{folding relation} between Dynkin diagrams of type~$\rE_6$ and~$\rF_4$
\begin{equation}
\label{diagram:folding}
  \dynkin[edge length = 2em,labels*={1,...,6},involutions={16;35}]E6
  \quad = \quad
  \dynkin[edge length = 2em,labels*={1,...,4}]F4.
\end{equation}
This provides a relation between representations of groups of type $\rE_6$ and~$\rF_4$,
and as a consequence, between their homogeneous varieties.
It takes the simplest form for the Cayley plane
\begin{equation*}
X = \rE_6/\rP_1 \subset \bP^{26}
\end{equation*}
and the coadjoint Grassmannian of Dynkin type~$\rF_4$
\begin{equation*}
Y = \rF_4/\rP_4 \subset \bP^{25}.
\end{equation*}
In this case $Y$ is isomorphic to a generic hyperplane section of~$X$, see~\cite[\S6.3]{LM} and~\cite[\S III.2.5.F]{Zak}.
The embedding $Y \hookrightarrow X$ allows us to use the results \cite{Manivel,FM} of Faenzi and Manivel on the derived category of~$X$
to describe the derived category of~$Y$.

More precisely, we do the following. The exceptional collection on~$X$ constructed in~\cite{FM} can be rewritten (see Theorem~\ref{theorem:fm}) as
\begin{multline}
\label{eq:dbx-intro}
\Db(X) = \langle
\cO_X, \cE_1^\vee, \cE_2^\vee;\
\cO_X(1), \cE_1^\vee(1), \cE_2^\vee(1);\
\cO_X(2), \cE_1^\vee(2), \cE_2^\vee(2);\ \\
\cO_X(3), \cE_1^\vee(3);\ \dots;\
\cO_X(11), \cE_1^\vee(11) \rangle,
\end{multline}
where $\cE_1$ and $\cE_2$ are equivariant vector bundles of ranks~10 and~54.
Note that this is a \emph{Lefschetz collection},
a notion introduced in~\cite{K07} in the context of homological projective duality,
see also~\cite{Ku08a,K14}, or~\S\ref{subsection:lefschetz} below.
In other words, it is split into blocks (separated in~\eqref{eq:dbx-intro} by semicolons,
so that we have three blocks of length~3 and nine blocks of length~2, altogether~27 bundles),
which are related via twists by~$\mathcal{O}_X(1)$.

Lefschetz collections are known to behave well with respect to hyperplane sections:
if one removes the first block and restricts the others to the hyperplane section,
we again obtain an exceptional collection, although not full in general, see~\cite[Proposition 2.4]{Ku08a}.
We show that in the case of the Lefschetz collection~\eqref{eq:dbx-intro} for the Cayley plane
the restricted collection on~$Y$ \emph{is} in fact full, hence we have a full exceptional collection
\begin{multline}
\label{eq:dby-intro}
\Db(Y) = \langle
\cO_Y, \cE_1^\vee\vert_Y, \cE_2^\vee\vert_Y;\
\cO_Y(1), \cE_1^\vee(1)\vert_Y, \cE_2^\vee(1)\vert_Y;\ \\
\cO_Y(2), \cE_1^\vee(2)\vert_Y;\ \dots;\
\cO_Y(10), \cE_1^\vee(10)\vert_Y \rangle.
\end{multline}
To prove this we verify in Proposition~\ref{proposition:bce2} that the vector bundle $\cE_2^\vee(2)\vert_Y$
belongs to the right-hand side of~\eqref{eq:dby-intro},
and then appeal to an argument of Samokhin from the proof of~\cite[Theorem 2.3]{Samokhin} (see Theorem~\ref{theorem:samokhin})
to check that equality holds in~\eqref{eq:dby-intro}. This proves Theorem~\ref{theorem:f4-intro}.

This result has a natural interpretation from the point of view of homological projective duality.
Indeed, it is equivalent to saying that the HPD part of the derived category of any smooth hyperplane section of the Cayley plane~$X$ vanishes.
This means that the homological projective dual variety is supported over the classical projective dual hypersurface of~$X$,
the \emph{Cartan cubic}, see~\cite{IM}.
Thus, it is natural to suggest the following

\begin{conjecture}
The homological projective dual variety of the Cayley plane
is a non-commutative resolution of the Cartan cubic hypersurface in~$\bP^{26}$.
\end{conjecture}

Note that the Cayley plane is one of the four \emph{Severi varieties} \cite{Zak}.
Each Severi variety has a singular cubic hypersurface as its classical projective dual variety,
and in the two out of the other three cases, for $\G(2,6)$ and~$\bP^2 \times \bP^2$,
the homological projective dual varieties are noncommutative resolutions of the corresponding cubic hypersurfaces,
the Pfaffian hypersurface in~$\bP^{14}$ and the determinantal hypersurface in~$\bP^8$,
see~\cite[Theorem~1]{0610957v1} and~\cite[Theorem~C.1, Remark~C.2]{K17}.
Thus, the Cayley plane is an analogue of these two simpler cases.

In the case of the last Severi variety, the double Veronese embedding of~$\bP^2$, the analogy partially fails ---
the homologically projective dual variety is given by a sheaf of noncommutative algebras over~$\bP^5$,
which are generically Morita-trivial and whose discriminant locus is the symmetric determinantal cubic hypersurface
(the classical projective dual of the Veronese surface), see~\cite{K08b}.

\medskip

The second result of this paper is a description of the residual categories
for the Lefschetz collections~\eqref{eq:dbx-intro} and~\eqref{eq:dby-intro}.

The notion of a residual category was introduced in~\cite{KS2020}, see~\S\ref{subsection:residual-general} for an overview.
It is an invariant of a full Lefschetz collection, which vanishes if and only if the collection is rectangular,
i.e., all blocks in the collection are the same.

In~\cite[Conjecture~1.12]{KS2020} a conjecture relating the structure of the residual category and the small quantum cohomology ring of a variety was suggested:
when the small quantum cohomology ring is generically semisimple,
it predicts that the residual category is generated by a completely orthogonal exceptional collection.
For the Cayley plane the small quantum cohomology is known to be generically semisimple~\cite[Corollary~1.2]{ChMaPe3},
and we verify that the prediction for the residual category indeed holds.

\begin{theorem}
\label{theorem:residual-e6-intro}
The residual category of the Lefschetz collection~\eqref{eq:dbx-intro} on the Cayley plane
is generated by three completely orthogonal exceptional bundles.
\end{theorem}

The small quantum cohomology of the coadjoint Grassmannian of type~$\rF_4$
was studied in~\cite[Proposition 5.3]{ChPe} and is known \emph{not} to be generically semisimple,
so~\cite[Conjecture~1.12]{KS2020} is not applicable here.
Still, a more elaborate conjecture~\cite[Conjecture~1.1]{KSnew} can be applied in such a situation.
In the case of coadjoint Grassmannians this conjecture was made more precise in~\cite[Conjecture~1.5]{KSnew};
and we verify that its prediction holds in type~$\rF_4$.

\begin{theorem}
\label{theorem:residual-f4-intro}
The residual category of the Lefschetz collection~\eqref{eq:dby-intro} on the coadjoint Grassmannian of type~$\rF_4$ is equivalent
to the derived category of the Dynkin quiver of type~$\rA_2$.
\end{theorem}

Precise versions of these two theorems can be found in Theorem~\ref{theorem:e6-residual}
and~Theorem~\ref{theorem:f4-residual} in the body of the paper.

We prove Theorem~\ref{theorem:residual-e6-intro} by an explicit calculation, see~\S\ref{subsection:residual-e6},
analogous to that of~\cite[Theorem~9.5]{CMKMPS}.
In particular, we describe explicitly the three exceptional bundles generating the residual category.

On the other hand, we deduce Theorem~\ref{theorem:residual-f4-intro} from Theorem~\ref{theorem:f4-intro}, Theorem~\ref{theorem:residual-e6-intro},
and a general result of independent interest, Theorem~\ref{theorem:restriction-general},
which says that if the derived category of a smooth projective variety~$X$ has a Lefschetz decomposition
whose residual category is generated by a completely orthogonal exceptional collection
and the derived category of a hyperplane section $Y \subset X$ is generated by the restricted Lefschetz collection,
then the residual category of~$Y$ is a product of derived categories of Dynkin quivers of type~$\rA$.

\medskip

To finish the introduction we want to point out that the similarity between the Cayley plane and the Severi variety $\G(2,6)$ (which is a homogeneous variety of Dynkin type~$\rA_5$)
persists on the level of residual categories and hyperplane sections:
\begin{itemize}
\item
The Grassmannian $\G(2,6)$ has a Lefschetz collection with three blocks of length~3 and three blocks of length~2 \cite[Theorem~4.1]{Ku08a}
and the residual category of~$\G(2,6)$ is generated by~3 completely orthogonal exceptional bundles~\cite[Theorem~9.5]{CMKMPS}.
\item
Smooth hyperplane sections $\IG(2,6) \subset \G(2,6)$ are homogeneous varieties
that correspond to the folding of the Dynkin diagram~$\rA_5$ into the Dynkin diagram~$\rC_3$:
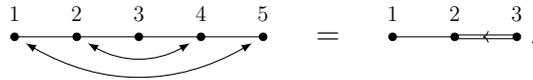
\begin{equation*}
  \dynkin[edge length=2em,labels*={1,...,5},involutions={15;24}]A5
  \quad
  =
  \quad
  \dynkin[edge length=2em,labels*={1,2,3}]C3.
\end{equation*}
\item
The restricted Lefschetz collection of~$\G(2,6)$ (with two blocks of length~3 and three blocks of length~2) on~$\IG(2,6)$
is full~\cite[Theorem~5.1]{Ku08a}
and the residual category of~$\IG(2,6)$ is equivalent to the derived category of representations
of the Dynkin quiver of type~$\rA_2$ \cite[Theorem~9.6]{CMKMPS}.
\end{itemize}

Of course, the last statement could be also deduced from Theorem~\ref{theorem:restriction-general}
in the same way as Theorem~\ref{theorem:residual-f4-intro}.

The paper is organized as follows.
In~\S\ref{section:lefschetz-residual-general} we recall definitions and basic facts about Lefschetz collections and their residual categories,
and prove Theorems~\ref{theorem:samokhin} and~\ref{theorem:restriction-general}.
In~\S\ref{section:residual-category-cayley-plane} we collect some facts about vector bundles on the Cayley plane
and prove Theorem~\ref{theorem:e6-residual}, which is a precise version of Theorem~\ref{theorem:residual-e6-intro}.
Finally, in~\S\ref{section:f4} we prove Theorems~\ref{theorem:f4} and~\ref{theorem:f4-residual}, which are
precise versions of Theorems~\ref{theorem:f4-intro} and~\ref{theorem:residual-f4-intro}.
In Appendix~\ref{appendix:weights} we collect some computations with weight lattices that are used in the body of the paper.

\medskip

\textbf{Conventions.}
Throughout we work over an algebraically closed base field~$\Bbbk$ of characteristic~0.
All functors between derived categories are implicitly derived.

\section{Lefschetz decompositions and residual categories}
\label{section:lefschetz-residual-general}

We start with a brief reminder of Lefschetz decompositions and their residual categories.
Then in Theorem~\ref{theorem:restriction-general} we prove, under appropriate assumptions, a general result relating the residual categories of a variety and a hyperplane section.

\subsection{Lefschetz decompositions}
\label{subsection:lefschetz}

Let $X \subset \bP(V)$ be a smooth projective variety.
A \emph{Lefschetz decomposition} \cite{K07} of~$\Db(X)$ is a semiorthogonal decomposition of the form
\begin{equation}
\label{eq:lefschetz-general}
\Db(X) = \langle \cA_0, \cA_1(1), \dots, \cA_{m-1}(m-1) \rangle,
\end{equation}
where the categories $\cA_p$ form a chain
\begin{equation}
\label{eq:lefschetz-chain}
0 \ne \cA_{m-1} \subseteq \dots \subseteq \cA_1 \subseteq \cA_0
\end{equation}
in~$\Db(X)$ and $\cA_p(t) = \cA_p \otimes \cO_{\bP(V)}(t)\vert_X$.
The integer~$m$ is called the \emph{length} of the Lefschetz decomposition.

If the first component $\cA_0$ of a Lefschetz decomposition is generated by an exceptional collection
\begin{equation*}
\cA_0 = \langle E_1, \dots, E_k \rangle
\end{equation*}
which is compatible with the Lefschetz chain~\eqref{eq:lefschetz-chain}, i.e.,
\begin{equation*}
\cA_p = \langle E_1, \dots, E_{k_p} \rangle
\end{equation*}
for a non-increasing sequence $k = k_0 \ge k_1 \ge \dots \ge k_{m-1} > 0$ we say that
\begin{equation*}
(E_1, \dots, E_{k_0};\ E_1(1), \dots, E_{k_1}(1);\ \dots;\ E_1(m-1), \dots, E_{k_{m-1}}(m-1))
\end{equation*}
is a \emph{Lefschetz collection} in $\Db(X)$.

More generally, if $\cD$ is an admissible subcategory of $\Db(X)$ for a smooth and proper variety $X$
(one could also work with smooth and proper DG-categories, but we do not need this level of generality)
and~$\tau \colon \cD \to \cD$ is an autoequivalence, a Lefschetz decomposition of~$\cD$ with respect to~$\tau$
is a semiorthogonal decomposition of the form
\begin{equation*}
\cD = \langle \cA_0, \tau(\cA_1), \dots, \tau^{m-1}(\cA_{m-1}) \rangle,
\end{equation*}
where the categories $\cA_p$ form a chain of admissible subcategories in~$\cD$, similar to~\eqref{eq:lefschetz-chain}.

Lefschetz decompositions of projective varieties are useful for many reasons, see~\cite{K14} for a survey.
One of the most important properties is their compatibility with hyperplane sections.

\begin{lemma}[{\cite[Proposition 2.4]{Ku08a}}]
\label{lemma:lefschetz-hyperplane}
Let~\eqref{eq:lefschetz-general} be a Lefschetz decomposition and
let~\mbox{$i \colon Y \hookrightarrow X$} be a hyperplane section.
The derived pullback functor
\begin{equation*}
i^* \colon \Db(X) \to \Db(Y)
\end{equation*}
is fully faithful on the components~$\cA_p$ for~\mbox{$p \ge 1$}.
Moreover, if
\begin{equation}
\label{eq:bca-general}
\bcA_p = i^*(\cA_p),
\qquad
1 \le p \le m - 1
\end{equation}
then the subcategories $\bcA_p(p-1)$, $1 \le p \le m-1$, are semiorthogonal in~$\Db(Y)$.
\end{lemma}

The lemma implies that there is an admissible subcategory \mbox{$\cD \subset \Db(Y)$} and a Lefschetz decomposition
\begin{equation}
\label{eq:restricted-lefschetz}
\cD = \langle \bcA_1, \bcA_2(1), \dots, \bcA_{m-1}(m-2) \rangle.
\end{equation}
We will call this the \emph{restricted Lefschetz decomposition}.
If $\cD = \Db(Y)$ we say that the restricted Lefschetz decomposition \emph{generates}~$\Db(Y)$.

The following result, proved in a particular case by Samokhin,
can be used to check that the restricted Lefschetz decomposition generates~$\Db(Y)$.
We denote by~$\langle \cF, \cG \rangle$ the triangulated subcategory generated by an exceptional pair~$(\cF,\cG)$.

\begin{theorem}[{cf.~\cite[Theorem~2.3]{Samokhin}}]
\label{theorem:samokhin}
Let $i\colon Y \hookrightarrow X$ be a closed embedding of a proper subvariety
and let $(\cF,\cG)$ be an exceptional pair of coherent sheaves on~$X$ with~$\cF$ torsion-free.
Consider the full subcategory
\begin{equation*}
\cC \coloneqq \{ C \in \Db(Y) \mid i_*C \in \langle \cF, \cG \rangle \}.
\end{equation*}
If the morphism of Grothendieck groups $\rK_0(\cC) \to \rK_0(Y)$
induced by the embedding of categories~$\cC \hookrightarrow \Db(Y)$ is zero, then~\mbox{$\cC = 0$}.
\end{theorem}

\begin{proof}
Take any object $C \in \cC {} \subset \Db(Y)$ and denote by $\cH^t(C)$ its cohomology sheaf in degree~$t \in \bZ$.
The decomposition triangle
\begin{equation*}
B^\bullet \otimes \cG \to i_*C \to A^\bullet \otimes \cF \, ,
\end{equation*}
for the exceptional pair $\langle\cF,\cG\rangle$,
where~$A^\bullet$ and~$B^\bullet$ are graded vector spaces,
gives a long exact sequence of cohomology sheaves
\begin{equation*}
\dots \to A^{t-1} \otimes \cF \to B^t \otimes \cG \to \cH^t(i_*C) \to A^t \otimes \cF \to \dots
\end{equation*}
The morphism $\cH^t(i_*C) \to A^t \otimes \cF$ is zero,
because $\cH^t(i_*C) \cong i_*\cH^t(C)$ is a torsion sheaf on~$X$ and~$\cF$ is torsion-free.
Therefore, the sequence splits into short exact sequences
\begin{equation*}
0 \to A^{t-1} \otimes \cF \to B^t \otimes \cG \to i_*\cH^t(C) \to 0.
\end{equation*}
In particular, it follows that $\cH^t(C) \in \cC$ for each~$t$.

If $\cH^t(C) \ne 0$ for some $t$ then for $n \gg 0$ the sheaf $\cH^t(C) \otimes \cO_Y(n)$ has zero higher cohomology
and nonzero global sections, hence $\cH^t(C)$ is not numerically orthogonal to~$\cO_Y(-n)$, hence its class in~$\rK_0(Y)$ is non-zero.
This contradicts the assumption about the Grothendieck groups, therefore $\cH^t(C) = 0$ for each~$t$, and hence $C = 0$.
\end{proof}

\subsection{Residual categories}
\label{subsection:residual-general}

Let $X$ be a smooth projective variety with~$\Db(X)$ endowed with a Lefschetz decomposition~\eqref{eq:lefschetz-general}.
Assume $\omega_X \cong \cO_X(-m)$, where $m$ is the length of~\eqref{eq:lefschetz-general}.

\begin{definition}[{\cite[Definition~2.7]{KS2020}}]
The \emph{residual category} of~\eqref{eq:lefschetz-general} is defined as
\begin{equation}
\label{eq:residual-general}
\cR = \langle \cA_{m-1}, \cA_{m-1}(1), \dots, \cA_{m-1}(m-1) \rangle^\perp.
\end{equation}
\end{definition}

The residual category is an admissible subcategory of~$\Db(X)$,
and measures the difference between the smallest block~$\cA_{m-1}$ and the others.

It was shown in~\cite[Theorem~2.8]{KS2020} that $\cR$ is endowed with a natural autoequivalence
\begin{equation}
\label{eq:tau-general}
\tau \colon \cR \to \cR,
\qquad
R \mapsto \bL_{\cA_{m-1}}(R(1)),
\end{equation}
called the \emph{induced polarization} of~$\cR$, where $\bL$ stands for the left mutation functor,
which enjoys the property $\tau^m \cong \bS_\cR^{-1}[\dim X]$, where $\bS_\cR$ is the Serre functor of~$\cR$;
this property is analogous to the relation between the twist functor and the Serre functor of~$\Db(X)$.

The following result relates Lefschetz decompositions of the residual category~$\cR$ with respect to the induced polarization~$\tau$
to Lefschetz decompositions of~$\Db(X)$.

\begin{proposition}[{\cite[Proposition~2.10]{KS2020}}]
\label{proposition:lefschetz-correspondence}
Let~$X$ be a smooth projective variety such that $\omega_X \cong \cO_X(-m)$,
and let~$\cB$ be an admissible subcategory of~$\Db(X)$ such that~
\begin{equation*}
(\cB,\cB(1),\ldots,\cB(m-1))
\end{equation*}
is a semiorthogonal collection of subcategories. Let~$\mathcal{R}$ be the residual category, i.e.
\begin{equation*}
\mathcal{R}\coloneqq\langle \cB, \cB(1),\ldots , \cB(m-1)\rangle^\perp.
\end{equation*}
Then there exists a bijection between the sets of
\begin{itemize}
\item Lefschetz decompositions of the category~$\mathcal{R}$ with respect to the induced polarisation~$\tau$; and
\item Lefschetz decompositions of the category~$\Db(X)$
with respect to~$\mathcal{O}_X(1)$
such that \mbox{$\cB \subseteq\mathcal{A}_{m-1}$}.
\end{itemize}
The bijection takes a Lefschetz decomposition
\begin{equation*}
\cR = \langle \cC_0, \tau(\cC_1), \dots, \tau^{m-1}(\cC_{m-1}) \rangle
\end{equation*}
to the Lefschetz decomposition~\eqref{eq:lefschetz-general} with~$\cA_p \coloneqq \langle \cC_p, \cB \rangle$.
\end{proposition}

\subsection{Residual categories of hyperplane sections}

As before, let~$X$ be a smooth projective variety with a Lefschetz decomposition~\eqref{eq:lefschetz-general}
and $\omega_X \cong \cO_X(-m)$.
Let $i \colon Y \hookrightarrow X$ be a smooth hyperplane section.
In general, the restricted Lefschetz decomposition~\eqref{eq:restricted-lefschetz} of~$\cD$ can be extended
to a Lefschetz decomposition of $\Db(Y)$
\begin{equation*}
\Db(Y) = \langle \bcA_1^+, \bcA_2(1), \dots, \bcA_{m-1}(m-2) \rangle
\end{equation*}
by replacing its first component~$\bcA_1$ by the bigger category
\begin{equation*}
\label{eq:bca-plus}
\bcA_1^+ \coloneqq \langle \bcA_2(1), \dots, \bcA_{m-1}(m-2) \rangle^\perp.
\end{equation*}
Let~\mbox{$\cR \subset \Db(X)$} and~\mbox{$\bcR \subset \Db(Y)$} be the corresponding residual categories.

\begin{lemma}
\label{lemma:residual-restriction}
Assume $m > 1$.
The restriction functor $i^* \colon \Db(X) \to \Db(Y)$ is compatible with the residual categories, i.e.,~$i^*(\cR) \subset \bcR$.
Moreover, the induced polarizations $\tau \colon \cR \to \cR$ and $\bar\tau \colon \bcR \to \bcR$
are related via the natural isomorphism
\begin{equation*}
i^*\circ\tau\cong\bar{\tau}\circ i^*.
\end{equation*}
\end{lemma}
\begin{proof}
Note that $\omega_X \cong \cO_X(-m)$ implies~$\omega_Y \cong \cO_Y(1-m)$.
Recall that $\cR$ is defined by~\eqref{eq:residual-general} and analogously
\begin{equation*}
\bcR = \langle \bcA_{m-1}, \bcA_{m-1}(1), \dots, \bcA_{m-1}(m-2) \rangle^\perp.
\end{equation*}
To prove the inclusion $i^*(\cR) \subset \bcR$ we must show that
if $R \in \cR$ then for any object~\mbox{$A \in \cA_{m-1}$} and any~\mbox{$0 \le t \le m-2$}
we have $\Hom_Y(i^*A(t),i^*R) = 0$.

Indeed, the pullback--pushforward adjunction and projection formula imply
\begin{equation}
\label{equation:want-vanishing}
\Hom_Y(i^*A(t),i^*R) \cong
\Hom_X(A(t),i_*i^*R) \cong
\Hom_X(A(t),R \otimes \cO_Y).
\end{equation}
Using the divisor short exact sequence
\begin{equation}
\label{eq:koszul-y}
0 \to \cO_X(-1) \to \cO_X \to \cO_Y \to 0
\end{equation}
the vanishing of the spaces in \eqref{equation:want-vanishing}
follows from the two vanishings
\begin{equation*}
\Hom_X(A(t),R) = 0,
\qquad
\Hom_X(A(t),R(-1)) \cong \Hom_X(A(t+1),R) = 0,
\end{equation*}
both of which hold by definition of~$\cR$.

For the second claim, we take any $R \in \cR$ and consider the mutation triangle
\begin{equation*}
A \to R(1) \to \tau(R)
\end{equation*}
for~$R(1)$, where $A \in \cA_{m-1}$.
Applying $i^*$ we obtain the triangle
\begin{equation*}
i^*A \to (i^*R)(1) \to i^*(\tau(R)).
\end{equation*}
Since $m > 1$; $i^*A \in i^*\cA_{m-1} = \bcA_{m-1}$
and $i^*(\tau(R)) \in i^*\cR \subset \bcR$,
it follows that this is the mutation triangle for $(i^*R)(1)$, hence $\bar\tau(i^*R) \cong i^*(\tau(R))$.
\end{proof}

The main result of this section is the following

\begin{theorem}
\label{theorem:restriction-general}
Assume that the restricted Lefschetz decomposition generates the category~$\Db(Y)$, i.e.,
\begin{equation}
\label{eq:dby-induced}
\Db(Y) = \langle \bcA_1, \bcA_2(1), \dots, \bcA_{m-1}(m-2) \rangle.
\end{equation}
Assume moreover that the residual category $\cR$ of~$\Db(X)$ is generated
by a completely orthogonal exceptional collection
\begin{equation}
\label{eq:fi-orthogonal}
\cR = \langle R_1,\dots,R_n \rangle,
\qquad
\Ext_{X}^\bullet(R_i,R_j) = 0
\quad
\forall i \ne j.
\end{equation}
Then there exists a partition~$n=\sum_{i=1}^rn_i$ with $n_i \ge 2$ for all~$i$ and an equivalence
\begin{equation*}
\bcR \cong \Db(\rA_{n_1-1}) \times \Db(\rA_{n_2-1}) \times \dots  \times \Db(\rA_{n_r-1}).
\end{equation*}
with a product of derived categories of quivers of Dynkin type~$\rA_{n_i - 1}$.
\end{theorem}

\begin{proof}
Any autoequivalence of a category generated by a \emph{completely orthogonal} exceptional collection
is a composition of a permutation and shifts of objects of the collection.
Applying this observation to the induced polarization~$\tau$ of~$\cR$
(see~\eqref{eq:tau-general}) we conclude that there is a set decomposition
\begin{equation*}
\{1,2,\dots,n\} = S_1 \sqcup S_2 \sqcup \dots \sqcup S_r
\end{equation*}
and a cyclic ordering of each of the sets $S_j$
(which encode the cycle type of the permutation) such that
\begin{equation}
\label{eq:tau}
\tau(R_{s_j}) \cong R_{s_j+1}[d_{s_j}], \qquad d_{s_j} \in \bZ,
\end{equation}
where we assume that $s_j \in S_j$ and $s_j + 1$ denotes the next element in the cyclic ordering of~$S_j$.
We set $n_j \coloneqq |S_j|$.
Note that for each~$j$ we can shift the objects~$R_{s_j}$ in such a way that~\mbox{$d_{s_j} = 0$} for all~\mbox{$s_j \in S_j$} but one.

By Proposition~\ref{proposition:lefschetz-correspondence} the Lefschetz decomposition \eqref{eq:lefschetz-general} of~$\Db(X)$
corresponds to a Lefschetz decomposition of $\cR$ with respect to~$\tau$.
Since any admissible subcategory of a category generated by a completely orthogonal exceptional collection
is generated by a subcollection, it follows that there is a linear ordering
\begin{equation*}
S_j = (s_{j,0}, s_{j,1}, \dots, s_{j,n_j-1})
\end{equation*}
compatible with the cyclic ordering defined above (i.e., $s_{j,p+1} = s_{j,p} + 1$) such that the induced Lefschetz decomposition takes the form
\begin{equation*}
\cR = \langle \cC_0, \tau(\cC_1), \dots, \tau^{m-1}(\cC_{m-1}) \rangle,
\end{equation*}
where for $0 \leq p \leq m-1$ we have
\begin{equation*}
\cC_p = \Big\langle \{ R_{s_{j,0}} \mid j \colon p < n_j \} \Big\rangle
\qquad\text{and}\qquad
\tau^p(\cC_p) = \Big\langle \{ R_{s_{j,p}} \mid j \colon p < n_j \} \Big\rangle.
\end{equation*}
Using Proposition~\ref{proposition:lefschetz-correspondence} we can rewrite~\eqref{eq:lefschetz-general} as
\begin{equation}
\label{eq:ca-j}
\cA_{p} = \langle \cC_{p}, \cA_{m-1} \rangle \quad \text{for $0 \leq p \leq m-1$}.
\end{equation}
By Lemma~\ref{lemma:lefschetz-hyperplane} the functor $i^*$ is fully faithful on~$\cA_p$ for $p \ge 1$, hence
\begin{equation*}
\bcA_p = \Big\langle \bcC_p, \bcA_{m-1} \Big\rangle,
\qquad 1 \le p \le m-1,
\end{equation*}
where
\begin{equation*}
\bcC_p = \Big\langle \{i^*{R}_{s_{j,0}} \mid j \colon p < n_j \} \Big\rangle.
\end{equation*}
Applying Proposition~\ref{proposition:lefschetz-correspondence} again, we conclude from~\eqref{eq:dby-induced} that
\begin{equation*}
\bcR = \langle \bcC_1, \bar{\tau}(\bcC_2), \dots, \bar{\tau}^{m-2}(\bcC_{m-1}) \rangle.
\end{equation*}
It will be convenient to replace this decomposition by
\begin{equation}
\label{eq:bcr}
\bcR = \langle \bar\tau(\bcC_1), \bar{\tau}^2(\bcC_2), \dots, \bar{\tau}^{m-1}(\bcC_{m-1}) \rangle,
\end{equation}
Moreover, by~\eqref{eq:tau} and Lemma~\ref{lemma:residual-restriction} we conclude that,
shifting the objects $R_{s_{j,p}}$ appropriately to kill the shifts $d_{s_{j,p}}$ for $0 \le p < n_j - 1$,
we have
\begin{equation}
\label{eq:tau-convention}
\bar\tau(i^*R_{s_{j,p}}) =
\begin{cases}
i^*R_{s_{j,p+1}}, & \text{if $0 \le p < n_j - 1$},\\
i^*R_{s_{j,0}}[d_j], & \text{if $p = n_j - 1$},
\end{cases}
\end{equation}
for some $d_j \in \bZ$, hence
\begin{equation}
\label{eq:bcc-p}
\bar\tau^p(\bcC_p) = \Big\langle \{i^*{R}_{s_{j,p}} \mid j \colon p < n_j \} \Big\rangle.
\end{equation}
In other words, $\bcR$ is generated by $i^*{R}_{s_{j,p}}$ with $1 \le j \le r$ and $1 \le p \le n_j - 1$.
It remains to compute $\Ext$'s between the objects~$i^*{R}_{s_{j,p}}$ and show that~$n_j \ge 2$ for all~$j$.

As in the proof of Lemma~\ref{lemma:residual-restriction}, by the adjunction and projection formula we have a distinguished triangle
\begin{equation*}
\Ext_X^\bullet(R_{s_{j,p}}(1), R_{s_{k,q}}) \to
\Ext_X^\bullet(R_{s_{j,p}}, R_{s_{k,q}}) \to
\Ext_Y^\bullet(i^*{R}_{s_{j,p}}, i^*{R}_{s_{k,q}}).
\end{equation*}
Let us compute the first two terms of this triangle.
\begin{enumerate}
  \item $\Ext_X^\bullet(R_{s_{j,p}}(1), R_{s_{k,q}})$: by the definition~\eqref{eq:tau-general} of $\tau$ there is a distinguished triangle
  \begin{equation*}
  A \to R_{s_{j,p}}(1) \to \tau(R_{s_{j,p}}),
  \end{equation*}
  where $A \in \cA_{m-1}$. Since by definition of~$\cR$ we have $\Ext_X^\bullet(\cA_{m-1},\cR) = 0$, we conclude that
  \begin{equation*}
  \Ext_X^\bullet(R_{s_{j,p}}(1), R_{s_{k,q}}) \cong
  \Ext_X^\bullet(\tau(R_{s_{j,p}}), R_{s_{k,q}}).
  \end{equation*}
  It follows from~\eqref{eq:tau-convention} that for $q \ge 1$ this $\Ext$-space is non-zero if and only if $j = k$ and $p = q - 1$.
  Moreover, in this case the space is $1$-dimensional and sits in degree~$0$.

  \item $\Ext_X^\bullet(R_{s_{j,p}}, R_{s_{k,q}})$: similarly this space is non-zero if and only if $j = k$ and $p = q$;
  and again, in this case the space is 1-dimensional and sits in degree~$0$.
\end{enumerate}

This proves that for $q \ge 1$ and any $p$ we have
\begin{equation}
\label{eq:ext-r-r}
\Ext_Y^\bullet(i^*{R}_{s_{j,p}}, i^*{R}_{s_{k,q}}) =
\begin{cases}
\Bbbk, & \text{if $j = k$, $p = q$},\\
\Bbbk[1], & \text{if $j = k$, $p = q - 1$},\\
0, & \text{otherwise}.
\end{cases}
\end{equation}
In particular, if $n_j = 1$ for some~$j$
then the object $i^*R_{s_{j,0}}$ is orthogonal to $i^*R_{s_{k,q}}$ with $q \ge 1$ and any~$k$,
hence by~\eqref{eq:bcr} and~\eqref{eq:bcc-p} to the entire category~$\bcR$.
Since by Lemma~\ref{lemma:residual-restriction} we have $i^*R_{s_{j,0}} \in \bcR$, it follows that~$i^*R_{s_{j,0}} = 0$.
On the other hand, since~$i$ is the embedding of a hyperplane section, it follows that $R_{s_{j,0}}$ has 0-dimensional support.
It is easy to see that for an exceptional object this is impossible; this proves that $n_j \ge 2$ for each~$j$.

Finally, it follows from~\eqref{eq:ext-r-r} that
the subcategory $\bcR_j \subset \bcR$ generated by~$i^*{R}_{s_{j,p}}$
with fixed~$j$ and $1 \le p \le n_j-1$ is equivalent to the derived category of the quiver $\rA_{n_j-1}$
(with the equivalence defined by sending the object~$i^*{R}_{s_{j,p}}[-2p]$ to the simple object of the $p$-th vertex of the quiver)
and that the subcategories~$\bcR_j$ and~$\bcR_k$ are completely orthogonal for~$j \ne k$.
Since by~\eqref{eq:bcr} and~\eqref{eq:bcc-p} the objects~$i^*{R}_{s_{j,p}}$ generate~$\bcR$, the theorem follows.
\end{proof}

\begin{remark}
\label{remark:simples}
It follows from~\eqref{eq:ext-r-r}
that for each $j$ the object $i^*{R}_{s_{j,0}}$ belongs to the component~$\bcR_j$ of the residual category~$\bcR$
and that under the constructed equivalence of categories~\mbox{$\bcR_j \cong \Db(\rA_{n_j-1})$}
it corresponds (up to shift) to the projective module of the first vertex of the quiver.

Moreover, one can identify the autoequivalence of~$\Db(\rA_{n_j-1})$ corresponding to~$\bar\tau$.
Recall from~\cite[Theorem~0.1(2) and Table~I]{MR1868359} that the group of autoequivalences of~$\Db(\rA_{n_j-1})$
is generated by the shift functor and the \emph{Auslander--Reiten translation $\tau_{\mathrm{AR}}$} that acts by
\begin{equation*}
\rS_{j,1} \mapsto \rS_{j,2} \mapsto \dots \mapsto \rS_{j,n_j-1} \mapsto \rP_{j,1}[-1],
\end{equation*}
where $\rS_{j,p}$ and $\rP_{j,p}$ are the simple and projective modules of the $p$-th vertex of the quiver~$\rA_{n_j-1}$.
Using \eqref{eq:tau-convention} we conclude that
\begin{equation*}
\bar\tau\vert_{\bcR_j} \cong \tau_{\mathrm{AR}} \circ [2].
\end{equation*}
\end{remark}

\begin{remark}
\label{remark:hyperplane-quantum}
We expect that an analogue of Theorem~\ref{theorem:restriction-general} exists for quantum cohomology,
where the role of the residual categories is played by certain decomposition factors of the small quantum cohomology ring,
denoted by $\kappa^{-1}(0)$ in~\cite[Conjecture~1.1]{KSnew}.
The derived categories of Dynkin quivers of type~$\rA$ in the residual category of a hyperplane section
should correspond to Milnor algebras of isolated hypersurface singularities of type~$\rA$ appearing in the quantum cohomology.
The hyperplane sections
\begin{equation*}
\IG(2,2n) \subset \G(2,2n), \,\,
\rF_4/\rP_4 \subset \rE_6/\rP_1
\quad\text{and}\quad
\Fl(1,n;n+1) \subset \bP^n \times \bP^n
\end{equation*}
provide some evidence for this expectation (see~\cite{CMKMPS,KSnew,Pe,PeSm}).
\end{remark}

\section{The Cayley plane}
\label{section:residual-category-cayley-plane}

Let $\rG$ be the simple simply connected algebraic group of Dynkin type~$\rE_6$.
Let~\mbox{$\rP_1 \subset \rG$} be the maximal parabolic subgroup associated with the first vertex of its Dynkin diagram,
where we use the following numbering
\begin{equation}
\label{diagram:e6}
\dynkin[edge length = 2em,labels*={1,...,6}]E6.
\end{equation}
Let
\begin{equation*}
X = \rG/\rP_1
\end{equation*}
be the \emph{Cayley plane}.
This is the smallest homogeneous variety of the group~$\rG$;
it is the \emph{\textup(co\textup)minuscule} Grassmannian of type~$\rE_6$.
We have
\begin{equation}
\label{eq:omega-x}
\omega_X \cong \cO_X(-12),
\qquad
\dim X = 16.
\end{equation}

The fundamental weights of the group~$\rG$ are denoted~$\omega_1,\ldots,\omega_6$ as in Appendix~\ref{appendix:weights}.
Let~\mbox{$V = V^{\omega_1}_\rG$} be the fundamental representation of~$\rG$ associated with the first vertex of the Dynkin diagram.
Then $\dim V = 27$ and
\begin{equation*}
X \subset \bP(V^\vee)
\end{equation*}
is the orbit under $\rG$ of the highest weight vector. We denote by $\cO_X(1)$ the line bundle on~$X$ corresponding to the above projective embedding.

\subsection{Vector bundles on \texorpdfstring{$X$}{X}}
\label{subsection:e6-vb}

In what follows we denote by $V^{\lambda}_{\rG}$ the irreducible representation of~$\rG$ with highest weight~$\lambda$.
Similarly, we denote by $\cU^{\lambda}$ the $\rG$-equivariant vector bundle on~$X$ associated with the irreducible representation
with the highest weight~$\lambda$ of the Levi group~$\rL$ of the parabolic~$\rP_1$.
Note that
\begin{equation*}
\cU^{t\omega_1} \cong \cO_X(t)
\end{equation*}
for all $t \in \bZ$.

Following~\cite{Manivel, FM} we consider the triple of irreducible $\rG$-equivariant vector bundles on~$X$
\begin{equation}
\label{def:cei}
\cE_0 = \cO_X,
\qquad
\cE_1 = (\cU^{\omega_6})^\vee,
\qquad
\cE_2 = (\cU^{2\omega_6})^\vee.
\end{equation}
Note that the rank of~$\cE_1$ is~10 and the rank of~$\cE_2$ is~54.
The main result of~\cite{FM} is the following construction of a full Lefschetz collection in~$\Db(X)$. Let
\begin{equation}
\label{def:cai}
\cA_p =
\begin{cases}
\langle \cO_X, \cE_1^\vee, \cE_2^\vee \rangle, 	& 0 \le p \le 2\\
\langle \cO_X, \cE_1^\vee \rangle, 		& 3 \le p \le 11.\\
\end{cases}
\end{equation}

The following result is essentially due to Faenzi--Manivel \cite{FM},
we just rearrange their exceptional collection slightly for convenience of further use.

\begin{theorem}[\cite{FM}]
\label{theorem:fm}
There is a semiorthogonal decomposition
\begin{equation}
\label{eq:cai}
\Db(X) = \langle \cA_0, \cA_1(1), \dots, \cA_{11}(11) \rangle.
\end{equation}
\end{theorem}

\begin{proof}
By~\cite{FM} the category~$\Db(X)$ has a full exceptional collection
\begin{multline*}
(\cE_2,\cE_1,\cO_X;\, \cE_2(1),\cE_1(1),\cO_X(1);\, \cE_2(2),\cE_1(2),\cO_X(2);
\cE_1(3),\cO_X(3);\, \dots;\, \cE_1(11),\cO_X(11)).
\end{multline*}
Using dualization we obtain the full exceptional collection
\begin{multline*}
(\cO_X(-11),\cE_1^\vee(-11);\, \dots;\, \cO_X(-3),\cE_1^\vee(-3); \\
\cO_X(-2),\cE_1^\vee(-2),\cE_2^\vee(-2);\, \cO_X(-1),\cE_1^\vee(-1),\cE_2^\vee(-1);\, \cO_X,\cE_1^\vee,\cE_2^\vee).
\end{multline*}
Mutating the first~18 bundles (i.e., the first line above)
to the right of the last~9 bundles (the second line) we get the full exceptional collection
\begin{multline*}
(\cO_X(-2),\cE_1^\vee(-2),\cE_2^\vee(-2);\, \cO_X(-1),\cE_1^\vee(-1),\cE_2^\vee(-1);\, \cO_X,\cE_1^\vee,\cE_2^\vee; \\
\cO_X(1),\cE_1^\vee(1);\, \dots;\, \cO_X(9),\cE_1^\vee(9)).
\end{multline*}
Finally, after the $\cO_X(2)$-twist we get~\eqref{eq:cai}.
\end{proof}

Now we discuss several properties of the bundles $\cE_1$ and $\cE_2$ that will become useful later.

\begin{lemma}
\label{lemma:ce1-embedding}
There is a $\rG$-equivariant embedding of vector bundles
\begin{equation*}
\cE_1 \hookrightarrow V \otimes \cO_X.
\end{equation*}
\end{lemma}
\begin{proof}
The bundle $\cE_1^\vee$ by definition corresponds to a dominant weight of~$\rG$, hence it is globally generated.
Moreover, by the Borel--Weil--Bott theorem
\begin{equation*}
\HH^0(X,\cE_1^\vee) =
\HH^0(X,\cU^{\omega_6}) =
V^{\omega_6}_\rG \cong
V^\vee,
\end{equation*}
and the latter isomorphism is a particular case of the general property of representations of~$\rG$ ---
the highest weight of the dual representation is obtained by the folding involution (pictured in~\eqref{diagram:folding}).
Thus, we have an equivariant epimorphism~\mbox{$V^\vee \otimes \cO_X \twoheadrightarrow \cE_1^\vee$}.
Dualizing it we obtain the required embedding.
\end{proof}

\begin{lemma}[\cite{Manivel,FM}]
We have a $\rG$-equivariant direct sum decomposition
\begin{equation}
\label{eq:ce2-sum}
\Sym^2\cE_1 \cong \cE_2 \oplus \cO_X(-1).
\end{equation}
The projection to the second summand of~\eqref{eq:ce2-sum} defines
a $\rG$-equivariant quadratic form $\bq \colon \Sym^2\cE_1 \to \cO_X(-1)$;
it induces $\rG$-equivariant isomorphisms
\begin{equation}
\label{eq:ce1-iso}
\cE_1^\vee \cong \cE_1(1)
\end{equation}
and
\begin{equation}
\label{eq:ce2-iso}
\cE_2^\vee \cong \cE_2(2).
\end{equation}
\end{lemma}
\begin{proof}
The decomposition~\eqref{eq:ce2-sum} follows from~\cite[(1)]{FM}.

To prove~\eqref{eq:ce1-iso} note that the embedding $\cO_X(-1) \hookrightarrow \Sym^2\cE_1$
induces a non-trivial $\rG$-equi\-va\-ri\-ant morphism $\cE_1^\vee(-1) \to \cE_1$.
But any non-trivial equivariant morphism between irreducible equivariant bundles is an isomorphism.

To prove~\eqref{eq:ce2-iso} we take the symmetric square of the isomorphism~\eqref{eq:ce1-iso} and use~\eqref{eq:ce2-sum}.
\end{proof}

Recall that there is a unique (up to rescaling) $\rG$-invariant cubic form
\begin{equation*}
\Ca \in \Sym^3V^\vee
\end{equation*}
(which is sometimes called the {\em Cartan cubic}),
see~\cite[\S2.1]{IM} and~\cite[Theorem~3.3]{kim2019explicit} for an explicit formula.
Below we often consider $\Ca$ as a symmetric trilinear form.

Consider the associated morphism of vector bundles
\begin{equation}
\label{def:cartan}
V \otimes \cO_{\bP(V)} \xrightarrow{\ \Ca\ } V^\vee \otimes \cO_{\bP(V)}(1)
\end{equation}
on~$\bP(V)$ defined at point $y \in \bP(V)$ by $v \mapsto \Ca(v,y) {} \in V^\vee$.
Note that the morphism~\eqref{def:cartan} is symmetrically self-dual (up to twist), since the trilinear form~$\Ca$ is symmetric. Furthermore, the map~\eqref{def:cartan} is generically invertible over~$\bP(V)$, see~\cite[proof of Proposition~2.5]{IM}.

Denote by
\begin{equation*}
\bh \in \HH^0(X \times \bP(V), \cO_X(1) \boxtimes \cO_{\bP(V)}(1))
\end{equation*}
the restriction of the equation of the universal hyperplane to
\begin{equation*}
X \times \bP(V) \subset \bP(V^\vee) \times \bP(V).
\end{equation*}

\begin{lemma}
Consider the product $X \times \bP(V)$.
The diagram
\begin{equation}
\label{diagram:cartan-bq}
\vcenter{\xymatrix@C=-2em{
&
V \otimes \cO_X \boxtimes \cO_{\bP(V)} \ar[rr]^-{\ \operatorname{id} \boxtimes \Ca\ } &&
V^\vee \otimes \cO_X \boxtimes \cO_{\bP(V)}(1) \ar[dr]
\\
\cE_1 \boxtimes \cO_{\bP(V)} \ar[ur] \ar[rr]^-{\bq \boxtimes \operatorname{id}} &&
\cE_1^\vee(-1) \boxtimes \cO_{\bP(V)} \ar[rr]^-{\bh} &&
\cE_1^\vee \boxtimes \cO_{\bP(V)}(1)
}}
\end{equation}
\textup(where the left diagonal arrow is the embedding of Lemma~\textup{\ref{lemma:ce1-embedding}}
and the right diagonal arrow is its dual\textup)
is commutative up to rescaling.
\end{lemma}
\begin{proof}
The morphisms in~\eqref{diagram:cartan-bq} are $\rG$-equivariant, hence so are their compositions.
Moreover, the compositions are symmetrically self-dual (because~$\Ca$ and $\bq$ are),
and using~\eqref{eq:ce2-sum} it is easy to see that the space
\begin{equation*}
\begin{aligned}
\HH^0(X \times \bP(V),\Sym^2\cE_1^\vee \boxtimes \cO_{\bP(V)}(1)) &\cong
\HH^0(X,\cU^{2\omega_6} \oplus \cO_X(1)) \otimes \HH^0(\bP(V), \cO_{\bP(V)}(1)) \\ &\cong
(V^{2\omega_6}_{\rG} \oplus V) \otimes V^\vee
\end{aligned}
\end{equation*}
of symmetrically self-dual maps $\cE_1 \boxtimes \cO_{\bP(V)} \to \cE_1^\vee \boxtimes \cO_{\bP(V)}(1)$
contains the trivial representation with multiplicity~1.
Therefore, the two compositions are proportional to each other.
\end{proof}

\begin{remark}
It is easy to show that the two compositions in~\eqref{diagram:cartan-bq} are non-zero;
in particular, rescaling~$\bq$ or~$\Ca$ appropriately we may assume that~\eqref{diagram:cartan-bq} is commutative on the nose.
\end{remark}

\begin{lemma}
\label{lemma:wedge2-ce1}
The bundle $\wedge^2\cE_1^\vee(1)$ belongs to the subcategory
\begin{equation*}
{}^\perp (\cA_0(-1)) = \langle \cA_1, \cA_2(1), \dots, \cA_{11}(10) \rangle.
\end{equation*}
\end{lemma}
\begin{proof}
This follows from \cite[Lemma~5]{FM}.
\end{proof}

\subsection{Residual category}
\label{subsection:residual-e6}

In this section we compute the residual category of the Lefschetz collection~\eqref{def:cai}.
We follow the strategy of~\cite[\S9]{CMKMPS}.
Recall from~\cite[Lemma 4]{FM} that there exists an exact sequence
\begin{multline}
\label{eq.:staircase-complex-e6-p1}
0 \to \cE_2(-1) \to V \otimes \cE_1(-1) \to W \otimes \cO_X(-1)
  \to W^\vee \otimes \cO_X \to  V^\vee \otimes \cE_1^\vee \to \cE_2^\vee \to 0,
\end{multline}
where
\begin{equation*}
W = \wedge^2V \oplus V^\vee.
\end{equation*}
Note that this complex is self-dual up to twist.

Define the sheaves $F_i$ for $0 \leq i \leq 2$ by the following left truncations
\begin{equation}
\label{eq.:left-resolutions-of-Fi}
  \begin{aligned}
    & 0 \to \cE_2(-1) \to V \otimes \cE_1(-1) \to W \otimes \cO_X(-1) \to F_0 \to 0, \\
    & 0 \to \cE_2(-1) \to V \otimes \cE_1(-1) \to F_1 \to 0, \\
    & 0 \to \cE_2(-1) \to F_2 \to 0
  \end{aligned}
\end{equation}
of~\eqref{eq.:staircase-complex-e6-p1}.
Since~\eqref{eq.:staircase-complex-e6-p1} is exact,
its right truncations provide alternative resolutions of the sheaves~$F_i$:
\begin{equation}
\label{eq.:right-resolutions-of-Fi}
\begin{aligned}
  & 0 \to F_0 \to W^\vee \otimes \cO_X \to  V^\vee \otimes \cE_1^\vee \to \cE_2^\vee \to 0, \\
  & 0 \to F_1 \to W \otimes \cO_X(-1) \to W^\vee \otimes \cO_X \to  V^\vee \otimes \cE_1^\vee \to \cE_2^\vee \to 0, \\
  & 0 \to F_2 \to V \otimes \cE_1(-1) \to W \otimes \cO_X(-1)
   \to W^\vee \otimes \cO_X \to  V^\vee \otimes \cE_1^\vee \to \cE_2^\vee \to 0.
\end{aligned}
\end{equation}
Note that it follows from~\eqref{eq.:right-resolutions-of-Fi} that the sheaves $F_i$ are locally free.

The key computation is given in the next lemma.

\begin{lemma}
\label{lemma:projection-e6}
Set $\cB = \langle \cO_X, \cE_1^\vee \rangle$.
We have
\begin{equation}
\label{eq:fi}
\bL_{\langle \cB, \dots, \cB(i) \rangle} (\cE_2^\vee(i)) = F_i (i) [2+i] \quad \text{for} \quad 0 \leq i \leq 2,
\end{equation}
where $\bL$ stands for the left mutation functor.
\end{lemma}

\begin{proof}
To prove~\eqref{eq:fi} it is enough to check the following two facts for $0 \le i \le 2$:
\begin{enumerate}
\item
The sheaf $F_i(i)$ lies in $\langle \cB, \dots, \cB(i) \rangle^\perp$.
\item
There is a morphism  $\cE_2^\vee(i) \to F_i(i)[2+i]$, whose cone lies in~\mbox{$\langle \cB, \dots, \cB(i) \rangle$}.
\end{enumerate}

To show the first assertion we note that we have isomorphisms~$\cE_2(-1) \cong \cE_2^\vee(-3)$ by~\eqref{eq:ce2-iso}
and~\mbox{$\cE_1(-1) \cong \cE_1^\vee(-2)$} by~\eqref{eq:ce1-iso}.
Therefore, \eqref{eq.:left-resolutions-of-Fi} gives the inclusions
\begin{align*}
F_0 &\in \langle \cA_0(-3),\cA_0(-2),\cA_0(-1) \rangle,\\
F_1(1) &\in \langle \cA_0(-2),\cA_0(-1) \rangle,\\
F_2(2) &\in \hphantom{\langle}\cA_0(-1).
\end{align*}
Now considering the twist of~\eqref{eq:cai} by $\cO_X(i-3)$ and then using the semiorthogonality
of the resulting decomposition, we obtain~(1).

To show the second fact we just use \eqref{eq.:right-resolutions-of-Fi} twisted by $\cO_X(i)$.
\end{proof}

Now we can state and prove a more precise version of Theorem~\ref{theorem:residual-e6-intro}.

\begin{theorem}
\label{theorem:e6-residual}
The residual category $\cR$ of~$\Db(X)$ with respect to the Lefschetz decomposition~\eqref{eq:cai}
is generated by the completely orthogonal exceptional triple
\begin{equation*}
(F_0,F_1(1),F_2(2))
\end{equation*}
of vector bundles defined by the exact sequences~\eqref{eq.:left-resolutions-of-Fi} or~\eqref{eq.:right-resolutions-of-Fi}.
Moreover, the induced polarization~$\tau$ of~$\cR$ acts by
\begin{equation}
\label{eq:tau-e6}
\tau(F_0) \cong F_1(1)[1],
\quad
\tau(F_1(1)) \cong F_2(2)[1],
\quad
\tau(F_2(2)) \cong F_0[2].
\end{equation}
In particular, $\tau^3 \cong [4]$.
\end{theorem}

\begin{proof}
Recall that the residual category is defined as the orthogonal of the rectangular part of the Lefschetz collection.
More precisely, if we write $\cB = \langle \cO_X, \cE_1^\vee \rangle$, then
\begin{equation}
\label{def:residual-e6}
\cR = \langle \cB, \cB(1) \dots, \cB(11) \rangle^\perp.
\end{equation}
Therefore, $\cR$ is generated by the projections to~$\cR$ of the objects $\cE_2^\vee$, $\cE_2^\vee(1)$, $\cE_2^\vee(2)$
from~\eqref{def:cai} that do not belong to the rectangular part, i.e.,
\begin{equation*}
\cR = \langle \bL_\cB(\cE_2^\vee), \bL_{\langle \cB, \cB(1) \rangle}(\cE_2^\vee(1)),
\bL_{\langle \cB, \cB(1), \cB(2) \rangle}(\cE_2^\vee(2)) \rangle.
\end{equation*}
Applying Lemma~\ref{lemma:projection-e6} we deduce the equality
\begin{equation*}
\cR = \langle F_0, F_1(1), F_2(2) \rangle.
\end{equation*}
The first two isomorphisms of~\eqref{eq:tau-e6} follow from~\eqref{eq:fi} and~\cite[(2.4)]{KS2020}
and the third isomorphism is evident, because~$F_2(2) \cong \cE_2(1) \cong \cE_2^\vee(-1)$
by~\eqref{eq.:left-resolutions-of-Fi} and~\eqref{eq:ce2-iso}. The identification~$\tau^3\cong[4]$ is then immediate.

Let us prove that the collection $(F_0, F_1(1), F_2(2))$ is completely orthogonal.
Indeed, it is an exceptional collection by construction, and applying the autoequivalence~$\tau$ and using~\eqref{eq:tau-e6},
we conclude that the collections $(F_1(1)[1], F_2(2)[1], F_0[2])$ and $(F_2(2)[2], F_0[3], F_1(1)[3])$ are exceptional as well,
hence the claim.
\end{proof}

\section{The coadjoint Grassmannian of type \texorpdfstring{$\mathrm{F}_4$}{F4}}
\label{section:f4}

Now let $\brG$ be the simple simply connected algebraic group of Dynkin type~$\rF_4$.
Let $\brP_4 \subset \brG$ be the maximal parabolic subgroup associated with the fourth vertex of its Dynkin diagram,
where we use the following numbering
\begin{equation}
\label{diagram:f4}
\dynkin[edge length = 2em,labels*={1,...,4}]F4
\end{equation}
Let
\begin{equation*}
Y = \brG/\brP_4
\end{equation*}
be the corresponding homogeneous variety;
it is the \emph{coadjoint} Grassmannian of type~$\rF_4$.
We have
\begin{equation}
\label{eq:omega-y}
\omega_Y \cong \cO_Y(-11),
\qquad
\dim Y = 15.
\end{equation}

As we mentioned in the Introduction,
$Y$ can be identified with a hyperplane section of~\mbox{$X \subset \bP(V^\vee)$} corresponding to a general point
\begin{equation*}
v_0 \in \bP(V).
\end{equation*}
In particular, the group~$\brG$ can be identified with (the connected component of) the stabilizer of~$v_0$ in~$\rG$.
We have natural inclusions
\begin{equation}
\label{eq:brg}
\brG \hookrightarrow \rG
\end{equation}
and
\begin{equation}
i \colon Y \hookrightarrow X;
\end{equation}
the inclusion~$i$ is $\brG$-equivariant.

\subsection{Vector bundles on \texorpdfstring{$Y$}{Y}}
\label{subsection:f4-vb}

We consider the triple of $\brG$-equivariant vector bundles
\begin{equation}
\label{def:bcei}
\bcE_0 = \cE_0\vert_Y = \cO_Y,
\qquad
\bcE_1 = \cE_1\vert_Y,
\qquad
\bcE_2 = \cE_2\vert_Y,
\end{equation}
where the bundles~$\cE_i$ were defined in~\eqref{def:cei}.

In what follows we denote by $V^{\lambda}_{\brG}$ the irreducible representation of the group~$\brG$ with the highest weight~$\lambda$.
Similarly, we denote by $\bcU^{\lambda}$ the $\brG$-equivariant vector bundle associated with the irreducible representation
with the highest weight~$\lambda$ of the Levi group~$\brL$ of the parabolic~$\brP_4$.
The fundamental weights of~$\brG$ are denoted~$\bom_1,\ldots,\bom_4$ as in Appendix~\ref{appendix:weights}.
Note that
\begin{equation*}
\bcU^{t\bom_4} \cong \cO_Y(t)
\end{equation*}
for all $t \in \bZ$.

For any $\rG$-dominant weight $\lambda$ the restriction $V_\rG^\lambda\vert_{\brG}$ is a representation of~$\brG$.
It is not irreducible in general; for instance, $V\vert_{\brG} \cong V_{\brG}^{\bom_4} \oplus \Bbbk$ by Lemma~\ref{lemma:restrictions}.
Similarly, for any $\rL$-dominant
weight $\lambda$ the restriction $\cU^\lambda\vert_Y$ is a~$\brG$-equivariant bundle, not irreducible in general.

Below we describe the irreducible factors of the bundles~$\bcE_1$ and~$\bcE_2$. Recall that a 3-term complex sitting in degrees $-1,0,1$ and exact in the first and last terms is called a \emph{monad}.

\begin{lemma}
\label{lemma:monads-bce1}
$(1)$
There is a monad
\begin{equation}
\label{eq:monad-ce1}
\cO_Y \to \bcE_1^\vee \to \cO_Y(1)
\end{equation}
whose middle cohomology is isomorphic to $\bcU^{\bom_3}(-1)$.

$(2)$
There is a monad
\begin{equation}
\label{eq:monad-s2-ce1}
\bcE_1^\vee \to \bcE_2^\vee \oplus \cO_Y(1) \oplus \cO_Y(1) \to \bcE_1^\vee(1)
\end{equation}
whose middle cohomology is isomorphic to $\bcU^{2\bom_3}(-2) \oplus \cO_Y(1)$.

$(3)$
There is a complex
\begin{equation}
\label{eq:monad-l2-ce1}
\cO_Y \to \bcE_1^\vee \to \wedge^2\bcE_1^\vee \oplus \cO_Y(1) \to \bcE_1^\vee(1) \to \cO_Y(2)
\end{equation}
whose only cohomology is isomorphic to $\bcU^{\bom_2}(-1) \oplus \bcU^{\bom_1}$.
\end{lemma}

\begin{proof}
By Lemma~\ref{lemma:restrictions} the bundle $\bcE_1^\vee$ has a $\brG$-equivariant filtration
with factors~$\cO_Y$, $\bcU^{\bom_3}(-1)$, and $\cO_Y(1)$ (precisely in this order).
This is equivalent to the first assertion of the lemma.

The complexes~\eqref{eq:monad-s2-ce1} and~\eqref{eq:monad-l2-ce1}
are obtained as the symmetric and exterior square of~\eqref{eq:monad-ce1};
moreover, \eqref{eq:ce2-sum} is used to rewrite the middle term of~\eqref{eq:monad-s2-ce1}
and Lemma~\ref{lemma:squares} to identify the cohomology.
\end{proof}

\begin{remark}
In~\cite{KuPo} a method to construct exceptional objects in derived categories of homogeneous varieties
of simple algebraic groups as iterated extensions of irreducible vector bundles was given.
This construction can be interpreted in terms of right mutations in the \emph{equivariant} derived category.
As a starting point it takes what is called ``an exceptional block~$B$'' of weights;
the resulting bundles are then extensions of irreducible bundles with weights in the block~$B$.
Here we want to point out that the sets of weights
\begin{equation*}
\{0,\ \bom_3 - \bom_4,\ \bom_4\}
\quad\text{and}\quad
\{0,\ \bom_3 - \bom_4,\ \bom_4,\ 2\bom_3 - 2\bom_4,\ \bom_3,\ 2\bom_4\}
\end{equation*}
that by Lemma~\ref{lemma:monads-bce1}(1) and~(2)
correspond to irreducible factors of the bundles~$\bcE_1^\vee$ and~$\bcE_2^\vee$ \emph{do not}
form exceptional blocks.
Therefore, it is unclear whether one could obtain these bundles by the construction of~\cite{KuPo}.
\end{remark}

Note that every representation of~$\brG$ is self-dual
(because the longest element of the Weyl group of type~$\rF_4$ acts on the weight lattice as~$-1$).
In the case of the representation $V\vert_{\brG}$ the self-duality isomorphism can be made explicit
by means of the trilinear form~$\Ca$ defined in~\S\ref{subsection:e6-vb}.
Recall that $v_0 \in V$ is the~$\brG$-invariant vector that defines the hyperplane cutting out~$Y$ in~$X$.

\begin{lemma}
\label{lemma:ca-0}
The quadratic form $\Ca_0 \coloneqq \Ca(v_0,-,-) \in \Sym^2V^\vee$ is $\brG$-invariant and non-degenerate.
\end{lemma}
\begin{proof}
The form~$\Ca_0$ is $\brG$-invariant because $\Ca$ is~$\rG$-invariant and $v_0$ is $\brG$-invariant by definition.
Non-degeneracy of $\Ca_0$ follows from generic non-degeneracy of~\eqref{def:cartan} since the point $v_0$ is general.
\end{proof}

Restricting the embedding $\cE_1 \hookrightarrow V \otimes \cO_X$ of Lemma~\ref{lemma:ce1-embedding} to~$Y$
we obtain a~$\brG$-equivariant embedding $\bcE_1 \hookrightarrow V \otimes \cO_Y$.

\begin{lemma}
\label{lemma:bce1-isotropic}
The subbundle $\bcE_1 \hookrightarrow V \otimes \cO_Y$ is $\Ca_0$-isotropic.
\end{lemma}
\begin{proof}
The assertion follows easily from commutativity of the diagram~\eqref{diagram:cartan-bq}
since the equation $\bh$ of the universal hyperplane section of~$X$
vanishes on the subvariety~\mbox{$Y \times \{v_0\} \subset X \times \bP(V)$}.
\end{proof}

The following lemma is crucial for what follows.

\begin{lemma}
\label{lemma:monads-v}
$(1)$ There is a monad
\begin{equation}
\label{eq:monad-v}
\bcE_1^\vee(-1) \to V \otimes \cO_Y \to \bcE_1^\vee
\end{equation}
whose middle cohomology is isomorphic to~$\bcU^{\bom_1}(-1)$.

$(2)$ There is a complex
\begin{multline}
\label{eq:monad-l2-v}
\Sym^2\bcE_1^\vee(-2) \to V \otimes \bcE_1^\vee(-1) \to
\Big(\wedge^2V \otimes \cO_Y\Big) \oplus \Big(\bcE_1^\vee \otimes \bcE_1^\vee(-1) \Big)
\to V \otimes \bcE_1^\vee \to \Sym^2\bcE_1^\vee
\end{multline}
whose only cohomology is isomorphic to~$\bcU^{\bom_2}(-2)$.
\end{lemma}

\begin{proof}
The first morphism in~\eqref{eq:monad-v} is defined in Lemma~\ref{lemma:bce1-isotropic}
(using the identification of bundles~$\bcE_1^\vee(-1) \cong \bcE_1$ of~\eqref{eq:ce1-iso}); in particular it is injective.
The second morphism in~\eqref{eq:monad-v} is the composition of the dual of the first
with~\mbox{$\Ca_0 \colon V \otimes \cO_Y \xrightarrow{\ \sim\ } V^\vee \otimes \cO_Y$}, in particular it is surjective.
The composition of these morphisms is zero because the subbundle $\bcE_1 \subset V \otimes \cO_Y$ is~$\Ca_0$-isotropic.
The description of the cohomology bundle of~\eqref{eq:monad-v}
follows from Lemma~\ref{lemma:restrictions-v} combined with Lemma~\ref{lemma:restrictions}.

The complex~\eqref{eq:monad-l2-v} is the exterior square of~\eqref{eq:monad-v}
and the description of its cohomology sheaf follows from Lemma~\ref{lemma:squares-2}.
\end{proof}

\subsection{Proof of Theorems \ref{theorem:f4-intro} and \ref{theorem:residual-f4-intro}}
\label{subsection:f4-proofs}

We define the subcategories $\bcA_{p} \subset \Db(Y)$ by
\begin{equation}
\label{def:bcai}
\bcA_p = i^*(\cA_{p}) =
\begin{cases}
\langle \cO, \bcE_1^\vee, \bcE_2^\vee \rangle, 	& 1 \le p \le 2,\\
\langle \cO, \bcE_1^\vee \rangle, 		& 3 \le p \le 11.\\
\end{cases}
\end{equation}

Applying Lemma~\ref{lemma:lefschetz-hyperplane} to~\eqref{eq:cai} we obtain the following

\begin{lemma}
\label{lemma:bcai-semiorthogonal}
The pullback functor $i^* \colon \Db(X) \to \Db(Y)$ is fully faithful on the categories~$\cA_p$ for $1 \le p \le 11$,
and the collection of subcategories
\begin{equation*}
\bcA_{1}, \bcA_{2}(1), \dots, \bcA_{11}(10) \subset \Db(Y)
\end{equation*}
is semiorthogonal.
In particular, the bundles $(\cO_Y,\bcE_1,\bcE_2)$ on~$Y$ form an exceptional triple.
\end{lemma}

We denote by
\begin{equation}
\label{def:cd}
\cD \coloneqq \langle \bcA_{1}, \bcA_{2}(1), \dots, \bcA_{11}(10) \rangle \subset \Db(Y)
\end{equation}
the subcategory generated by the $\bcA_p(p-1)$ with $1 \leq p \leq 11$. To prove Theorem~\ref{theorem:f4-intro} we need to show that~$\cD = \Db(Y)$. We start by showing that some particular equivariant bundles belong to~$\cD$.

\begin{lemma}
\label{lemma:w2bce1}
We have $\wedge^2\bcE_1^\vee(1) \in \cD$.
\end{lemma}

\begin{proof}
We will use Lemma~\ref{lemma:wedge2-ce1} for this.
Indeed, by definition of a semior\-thogonal decomposition we have a sequence of morphisms in~$\Db(X)$
\begin{equation*}
0 = \cF_{11} \to \cF_{10} \to \dots \to \cF_0 = \wedge^2\cE_1^\vee(1),
\end{equation*}
such that $\Cone(\cF_p \to \cF_{p-1}) \in \cA_p(p-1)$, where $1 \le p \le 11$.
Applying the functor~$i^*$ we get a sequence of morphisms in $\Db(Y)$
\begin{equation*}
0 = i^*\cF_{11} \to i^*\cF_{10} \to \dots \to i^*\cF_0 = i^*(\wedge^2\cE_1^\vee(1)) = \wedge^2\bcE_1^\vee(1),
\end{equation*}
such that $\Cone(i^*\cF_p \to i^*\cF_{p-1}) \in i^*\cA_p(p-1) = \bcA_{p}(p-1)$.
This means that we have the required containment~$\wedge^2\bcE_1^\vee(1) \in \cD$, which finishes the proof.
\end{proof}

\begin{lemma}
\label{lemma:omega1-omega2}
We have
\begin{equation*}
\bcU^{\bom_1}(1), \bcU^{\bom_2} \in \cD.
\end{equation*}
\end{lemma}
\begin{proof}
The first containment
follows from the monad~\eqref{eq:monad-v} twisted by~$\cO_Y(2)$.
The second containment follows from~\eqref{eq:monad-l2-ce1} twisted by~$\cO_Y(1)$
combined with Lemma~\ref{lemma:w2bce1} and the first containment.
\end{proof}

\begin{proposition}
\label{proposition:bce2}
We have
\begin{equation*}
\bcE_2^\vee(2) \in \cD.
\end{equation*}
\end{proposition}
\begin{proof}
Consider the complex~\eqref{eq:monad-l2-v} twisted by~$\cO_Y(2)$.
By Lemma~\ref{lemma:omega1-omega2} its cohomology~$\bcU^{\bom_2}$ is contained in the category~$\cD$.
Moreover, by Lemma~\ref{lemma:w2bce1} (note~$\bcE_1^\vee \otimes \bcE_1^\vee \cong \wedge^2\bcE_1^\vee \oplus \bcE_2^\vee \oplus \cO_Y(1)$)
its first four terms are in~$\cD$.
Therefore, its last term $\Sym^2\bcE_1(2) \cong \bcE_2^\vee(2) \oplus \cO_Y(3)$ is also in~$\cD$.
Since~\mbox{$\cO_Y(3) \in \cD$}, we conclude that~$\bcE_2^\vee(2)$ is in~$\cD$ as well.
\end{proof}

\begin{remark}
In fact, the decomposition of $\bcE_2^\vee(2)$ with respect to the right-hand side of~\eqref{def:cd} can be made more precise:
one can show that there is an exact sequence
\begin{equation*}
0 \to \bcE_2^\vee \to V \otimes \bcE_1^\vee(1) \to
W \otimes \cO_Y(2) \oplus \bcE_2^\vee(1) \to
V \otimes \bcE_1^\vee(2) \to \bcE_2^\vee(2) \to 0,
\end{equation*}
where $W = \wedge^2V \oplus V$.
It is instructive to compare this with~\eqref{eq.:staircase-complex-e6-p1}.
\end{remark}

Now we are ready to prove the following more precise version of Theorem~\ref{theorem:f4-intro}.
Recall from~\eqref{def:bcai} the definition of the categories~$\bcA_p$.

\begin{theorem}
\label{theorem:f4}
We have a Lefschetz decomposition
\begin{equation}
\label{eq:dby-lefschetz}
\Db(Y) = \langle \bcA_1, \bcA_2(1), \dots, \bcA_{11}(10) \rangle.
\end{equation}
In particular, $\Db(Y)$ is generated by the exceptional collection~\eqref{eq:dby-intro} of length~$24$.
\end{theorem}

\begin{proof}

By~\eqref{def:cd} we need to show that $\cD = \Db(Y)$.
By Lemma~\ref{lemma:bcai-semiorthogonal} the category~$\cD$ is generated by an exceptional collection,
hence it is admissible, so we obtain a semiorthogonal decomposition
\begin{equation*}
\Db(Y) = \langle \cD^\perp, \cD \rangle,
\end{equation*}
and so we need to show that~\mbox{$\cD^\perp = 0$}.
We will deduce this from Theorem~\ref{theorem:samokhin}.

First, note that
\begin{equation*}
\rK_0(Y) = \rK_0(\cD) \oplus \rK_0(\cD^\perp).
\end{equation*}
Since $Y$ is a homogeneous variety, the left-hand side is a free abelian group
of rank equal to the index of the Weyl group of~$\brL$ in the Weyl group of~$\brG$, which is equal to~24.
On the other hand, the first summand in the right-hand side is also a free abelian group of rank~24,
because~$\cD$ is generated by~24 exceptional objects.
Therefore,
\begin{equation*}
\rK_0(\cD^\perp) = 0.
\end{equation*}

Let us show that
\begin{equation}
\label{eq:cd-perp}
\cD^\perp = \{ C \in \Db(Y) \mid i_*C \in \langle \cO_X(-1), \cE_1^\vee(-1) \rangle \}.
\end{equation}
Indeed, assume $C \in \cD^\perp$.
By definition of~$\cD$ we have
\begin{equation*}
\Hom_Y(i^*\cA_p(p-1), C) = 0
\qquad \text{for $1 \le p \le 11$}
\end{equation*}
and then by adjunction we obtain
\begin{equation*}
\Hom_X(\cA_p(p-1), i_*C) = 0
\qquad \text{for $1 \le p \le 11$.}
\end{equation*}
Then~\eqref{eq:cai} implies that $i_*C \in \cA_0(-1) = \langle \cO_X(-1), \cE_1^\vee(-1), \cE_2^\vee(-1) \rangle$.
Moreover, since we have $i^*\cE_2^\vee(2) = \bcE_2^\vee(2) \in \cD$ by Lemma~\ref{proposition:bce2},
the same argument proves that
\begin{equation*}
\Hom_X(\cE_2^\vee(2), i_*C) = 0.
\end{equation*}
Using~\eqref{eq.:staircase-complex-e6-p1} and the isomorphisms~\eqref{eq:ce1-iso} and~\eqref{eq:ce2-iso}, we deduce
\begin{equation*}
\Hom_X(\cE_2^\vee(-1), i_*C) = 0.
\end{equation*}
This proves that $i_*C \in \langle \cO_X(-1), \cE_1^\vee(-1) \rangle$.

Conversely, if $i_*C \in \langle \cO_X(-1), \cE_1^\vee(-1) \rangle \subset \cA_0(-1)$, then
using~\eqref{eq:cai} and the adjunction as above, we deduce that $C \in \cD^\perp$.

This shows that the assumptions of Theorem~\ref{theorem:samokhin} are satisfied for the category~$\cC = \cD^\perp$
and the exceptional pair $(\cO_X(-1), \cE_1^\vee(-1))$.
Therefore $\cD^\perp = 0$.
\end{proof}

Now we can also give a proof of the following more precise version of Theorem~\ref{theorem:residual-f4-intro}.
Recall the bundles $F_i$ defined by~\eqref{eq.:left-resolutions-of-Fi} or~\eqref{eq.:right-resolutions-of-Fi}.

\begin{theorem}
\label{theorem:f4-residual}
The residual category~$\bcR$ of~$\Db(Y)$ is generated by the exceptional pair of vector bundles $(i^*F_1(1),i^*F_2(2))$.
Moreover, there is an equivalence
\begin{equation*}
\bcR \cong \Db(\rA_2)
\end{equation*}
with the derived category of the quiver of Dynkin type~$\rA_2$,
such that the objects~$i^*F_1(1)[-1]$ and $i^*F_2(2)[-2]$ correspond to the simple modules,
and the induced polarization~$\bar\tau$ acts as the Auslander--Reiten translation composed with the shift by~$2$.
\end{theorem}

\begin{proof}
We apply Theorem~\ref{theorem:restriction-general} to the Cayley plane~$X$ and its hyperplane section~$Y$.
By Theorem~\ref{theorem:f4} the restricted Lefschetz decomposition generates~$\Db(Y)$
and by Theorem~\ref{theorem:e6-residual} the residual category~$\cR$ of~$\Db(X)$
is generated by a completely orthogonal exceptional collection~$(F_0,F_1(1),F_2(2))$.
Therefore, the residual category~$\bcR$ of~$\Db(Y)$ is equivalent to a product of derived categories of Dynkin quivers of type~$\rA$.
It remains to understand the types of these quivers explicitly.

As the proof of Theorem~\ref{theorem:restriction-general} shows the types are encoded
in the action of the induced polarization~$\tau$ of~$\cR$ on the set of objects $(F_0,F_1(1),F_2(2))$.
Using~\eqref{eq:tau-e6} we conclude that $\bcR \cong \Db(\rA_2)$
and that the simple modules of the path algebra of the quiver~$\rA_2$ correspond to the objects~$i^*F_1(1)[-1]$ and~$i^*F_2(2)[-2]$.
The final claim follows from Remark~\ref{remark:simples}.
\end{proof}

\appendix

\section{Computations in weight lattices}
\label{appendix:weights}

\subsection{Restriction of weights}
\label{appendix:restrictions}

Consider the commutative diagram of simple algebraic groups
\begin{equation*}
\xymatrix@R=3ex{
\rE_6 &&
\rF_4 \ar[ll]
\\
\rD_5 \ar[u] &
\rD_4 \ar[l] &
\rB_3. \ar[l] \ar[u]
}
\end{equation*}
Here the upper horizontal arrow is the embedding $\brG \hookrightarrow \rG$,
as the stabilizer of a general vector $v_0 \in V_\rG^{\omega_1}$;
the vertical arrows are the embeddings of the semisimple parts of the Levi groups~$\rL$ and~$\brL$
of the parabolic subgroups~$\rP_1 \subset \rG$ and~$\brP_4 \subset \brG$, respectively
(their Dynkin diagrams are obtained by removing the vertices~1 and~4 from~\eqref{diagram:e6} and~\eqref{diagram:f4}, respectively);
and the arrows in the bottom row are the standard embeddings $\Spin(7) \hookrightarrow \Spin(8) \hookrightarrow \Spin(10)$.

The morphisms on weight lattices induced by the above morphisms of groups can be described by the commutative diagram
\begin{equation}
\label{diagram:five-groups}
\vcenter{\xymatrix@R=3ex{
\dynkin[edge length = 2em,labels*={1,...,6}]E6
\ar@{=>}[rr] \ar[d] &&
\dynkin[edge length = 2em,labels*={2,4,{3,5},{1,6}}]F4
\ar[d]
\\
\dynkin[edge length = 2em, labels*={,2,...,6}]E{o*****}
\ar[r] &
\dynkin[edge length = 2em, labels*={,2,...,5,}]E{o****o}
\ar@{=>}[r] &
\dynkin[edge length = 2em,labels*={2,4,{3,5},}]F{***o}
}}
\end{equation}
where the labels $i_1,\dots,i_k$ on a vertex of a diagram mean that the fundamental weights $\omega_{i_1}$, \dots, $\omega_{i_k}$ of~$\rE_6$
go to the fundamental weight of this vertex;
moreover, if a label $i$ does not appear on a diagram, then the fundamental weight~$\omega_i$ of~$\rE_6$ goes to zero.
For example, the morphism from the weight lattice of $\rE_6$ to that of~$\rF_4$ is given by
\begin{align}
\label{eq:e6-f4}
(\omega_1,\omega_2,\omega_3,\omega_4,\omega_5,\omega_6)
&\mapsto
(\bom_4, \bom_1, \bom_3, \bom_2, \bom_3, \bom_4).
\end{align}
The double arrows in the diagram mean that the corresponding maps are foldings~\cite[\S3.6]{FoGo}.

We use the notations of~\S\ref{subsection:e6-vb} and~\S\ref{subsection:f4-vb} for representations and equivariant vector bundles.
Using diagram~\eqref{diagram:five-groups} it is easy to prove the following

\begin{lemma}
\label{lemma:restrictions}
We have $V_\rG^{\omega_1}\vert_{\brG} = V_{\brG}^{\bom_4} \oplus \Bbbk$.
Moreover, $\cU^{\omega_6}\vert_Y$ has a $\brG$-equivariant filtration with factors~$\cO_Y$, $\bcU^{\bom_3}(-1)$, $\cO_Y(1)$.
\end{lemma}

\begin{proof}
The weights in the $27$-dimensional representation $V = V_\rG^{\omega_1}$ are
\begin{equation*}
\omega_1,\ \omega_3 - \omega_1,\ \dots,\ \omega_3 - \omega_5,\ \dots,\ \omega_6 - \omega_5,\ -\omega_6
\end{equation*}
(see~\eqref{diagram:e6-omega1} for a more detailed picture).
By~\eqref{eq:e6-f4} the weights of its restriction to~$\brG$ are
\begin{equation*}
\bom_4, \bom_3 - \bom_4, \dots, 0, \dots, \bom_4 - \bom_3, -\bom_4.
\end{equation*}
In particular, the highest weight $\omega_1$ of $V$ goes to the weight~$\bom_4$,
which is hence the highest weight of the principal irreducible summand of~$V\vert_{\brG}$.
Since the dimension of this summand is $\dim(V_{\brG}^{\bom_4}) = 26$,
it follows that the other summand is 1-dimensional, hence trivial.
Its highest weight~$0$ is the image of the weight $\omega_3 - \omega_5$ in~$V$.

Similarly, the bundle $\cU^{\omega_6}$ corresponds to the representation~$V_{\rL}^{\omega_6}$ of the Levi group $\rL$
of Dynkin type~$\rD_5$
with the highest weight~$\omega_6$.
The diagram of~$\rL$-weights of this representation is
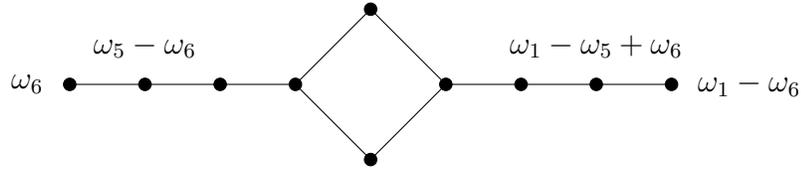
\begin{equation*}
\begin{tikzpicture}[xscale = 1, yscale = 1]
\node (-1,0) {} -- (11,0) {}; 
\draw (0,1) node[left=.5em]{$\omega_6$} --
      (1,1) node[above=.5em]{$\omega_5 - \omega_6$} --
      (2,1) -- (3,1) -- (4,2) -- (5,1) -- (6,1) --
      (7,1) node[above=.5em]{$\omega_1 - \omega_5 + \omega_6$} --
      (8,1) node[right=.5em]{$\omega_1 - \omega_6$};
\draw (3,1) -- (4,0) -- (5,1);
\foreach \i in {0,1,2,3,5,6,7,8} \filldraw[black] (\i,1) circle (.2em);
\filldraw[black] (4,2) circle (.2em);
\filldraw[black] (4,0) circle (.2em);
\end{tikzpicture}
\end{equation*}
%
By~\eqref{eq:e6-f4} the diagram of weights of its restriction to~$\brL$ is
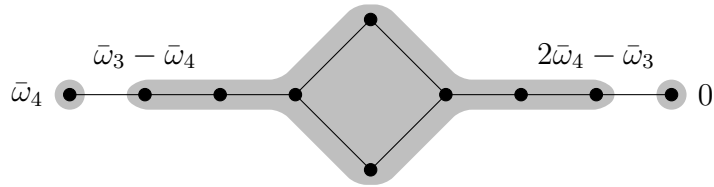
\begin{equation*}
\begin{tikzpicture}[xscale = 1, yscale = 1]
\node (-1,0) {} -- (11,0) {}; 
\fill[color = lightgray] (0,1) circle (.2);
\fill[color = lightgray] (8,1) circle (.2);
\fill [color = lightgray, rounded corners] (2, 1.2) -- (0.9, 1.2) -- (0.7, 1) -- (0.9, 0.8) --
      (2.8, 0.8) -- (3, 0.6) -- (3.8, -0.2) -- (4.2, -0.2) -- (5, 0.6) -- (5.2, 0.8) --
      (7.1, 0.8) -- (7.3, 1) -- (7.1, 1.2) --
      (5.2, 1.2) -- (5, 1.4) -- (4.2, 2.2) -- (3.8, 2.2) -- (3, 1.4) -- (2.8, 1.2) --
      (2, 1.2);
\draw (0,1) node[left=.5em]{$\bom_4$} --
      (1,1) node[above=.5em]{$\bom_3 - \bom_4$} --
      (2,1) -- (3,1) -- (4,2) -- (5,1) -- (6,1) --
      (7,1) node[above=.5em]{$2\bom_4 - \bom_3$} --
      (8,1) node[right=.5em]{$0$};
\draw (3,1) -- (4,0) -- (5,1);
\foreach \i in {0,1,2,3,5,6,7,8} \filldraw[black] (\i,1) circle (.2em);
\filldraw[black] (4,2) circle (.2em);
\filldraw[black] (4,0) circle (.2em);
\end{tikzpicture}
\end{equation*}%
%
The only~$\rB_3$-dominant (i.e., those with $\bom_1$, $\bom_2$, and $\bom_3$ appearing with non-negative coefficients)
weights in this list are $\bom_4$, $\bom_3 - \bom_4$, and~$0$;
the corresponding representations of~$\rB_3$, marked in gray,
have dimensions~$1$, $8$, and~$1$, that sum up to~$10$, the rank of~$\cU^{\omega_6}$.
This means that the corresponding vector bundles $\bcU^{\bom_4} = \cO_Y(1)$,
$\bcU^{\bom_3 - \bom_4} = \bcU^{\bom_3}(-1)$, and $\bcU^0 = \cO_Y$
are the factors of a $\brG$-equivariant filtration of~$\cU^{\omega_6}\vert_Y$.
\end{proof}

A similar argument proves the following

\begin{lemma}
\label{lemma:restrictions-v}
The trivial vector bundle $V\otimes \cO_Y$ has a $\brG$-equivariant filtration
with factors~$\cO_Y(-1)$, $\bcU^{\bom_3}(-2)$, $\cO_Y$, $\cO_Y$, $\bcU^{\bom_1}(-1)$, $\bcU^{\bom_3}(-1)$, $\cO_Y(1)$.
\end{lemma}

\begin{proof}
The next picture shows the diagram of $\rG$-weights of $V = V_\rG^{\omega_1}$:
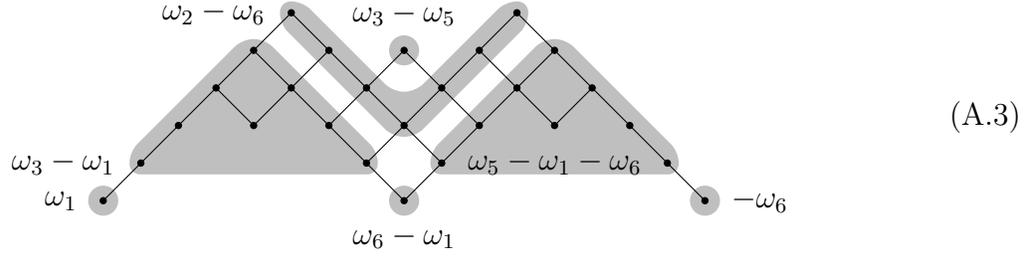
\begin{equation}
\label{diagram:e6-omega1}
{\begin{tikzpicture}[xscale = .5, yscale = .5, baseline=(current bounding box.center)]
\fill [color = lightgray, rounded corners] (4,.7) -- (.7,.7) -- (.7,1.3) -- (3.7,4.3) -- (4.3,4.3) -- (7.3,1.3) -- (7.3,.7) -- (4,.7);
\fill [color = lightgray, rounded corners] (12,.7) -- (8.7,.7) -- (8.7,1.3) -- (11.7,4.3) -- (12.3,4.3) -- (15.3,1.3) -- (15.3,.7) -- (12,.7);
\fill [color = lightgray, rounded corners] (6.3,4.3) -- (5.3,5.3) -- (4.7,5.3) -- (4.7,4.7) -- (7.7,1.7)-- (8.3,1.7) -- (11.3,4.7) -- (11.3,5.3) --
			(10.7,5.3) -- (8.3,2.9) -- (7.7,2.9) -- (6.3,4.3);
\fill[color = lightgray] (0,0) circle (.4);
\fill[color = lightgray] (8,0) circle (.4);
\fill[color = lightgray] (8,4) circle (.4);
\fill[color = lightgray] (16,0) circle (.4);
\draw (0,0) node[left=.5em]{$\omega_1$} --
      (1,1) node[left=.5em]{$\omega_3 - \omega_1$} --
      (2,2) --
      (3,3) --
      (4,4) --
      (5,5) node[left=.5em]{$\omega_2 - \omega_6$} --
      (6,4) --
      (7,3) --
      (8,4) node[above=1ex]{$\omega_3 - \omega_5$} --
      (9,3) --
      (10,4) --
      (11,5) --
      (12,4) --
      (13,3) --
      (14,2) --
      (15,1) --
      (16,0) node[right=.5em]{$-\omega_6$};
\draw (3,3) -- (4,2) -- (5,3) -- (6,4);
\draw (4,4) -- (5,3) -- (6,2) -- (7,3) -- (8,2) -- (9,3) -- (10,2) -- (11,3) -- (12,4);
\draw (10,4) -- (11,3) -- (12,2) -- (13,3);
\draw (6,2) -- (7,1) -- (8,2) -- (9,1) -- (10,2);
\draw (7,1) --
      (8,0) node[below = 1ex]{$\omega_6 - \omega_1$} --
      (9,1) node[right = .5em]{$\omega_5 - \omega_1 - \omega_6$};
\foreach \i in {0,1,2,3,4,5} \filldraw[black] (\i,\i) circle (.2em);
\foreach \i in {0,1,2} \filldraw[black] (\i + 4,\i + 2) circle (.2em);
\foreach \i in {0,1,2} \foreach \j in {0,1,2} \filldraw[black] (\i + \j + 6,\i - \j + 2) circle (.2em);
\foreach \i in {0,1,2} \filldraw[black] (12 - \i,2 + \i) circle (.2em);
\foreach \i in {0,1,2,3,4,5} \filldraw[black] (16 - \i,\i) circle (.2em);
\end{tikzpicture}}
\end{equation}
The labeled weights are the only $\rG$-weights of $V$ that project to~$\rB_3$-dominant weights of~$\brG$.
By~\eqref{eq:e6-f4} the resulting weights of~$\brG$ are
\begin{equation*}
\bom_4,\
\bom_3 - \bom_4,\
\bom_1 - \bom_4,\
0,\
0,\
\bom_3 - 2\bom_4,\
-\bom_4,
\end{equation*}
and these provide the required factors of the $\brG$-equivariant filtration, marked with gray.
\end{proof}

\subsection{Symmetric and exterior squares of representations}
\label{appendix:squares}

We need the following

\begin{lemma}
\label{lemma:squares}
We have isomorphisms
\begin{equation*}
\Sym^2\bcU^{\bom_3} \cong \bcU^{2\bom_3} \oplus \cO_Y(3)
\qquad\text{and}\qquad
\wedge^2\bcU^{\bom_3} = \bcU^{\bom_2}(1) \oplus \bcU^{\bom_1}(2).
\end{equation*}
\end{lemma}

\begin{proof}
The weights of the representation of~$\brL$ with the highest weight~$\bom_3$ are
\begin{multline*}
\bom_3,\
\bom_2 - \bom_3 + \bom_4,\
\bom_1 - \bom_2 + \bom_3 + \bom_4,\
-\bom_1 + \bom_3 + \bom_4,\\
\bom_1 - \bom_3 + 2\bom_4,\
-\bom_1 + \bom_2 - \bom_3 + 2\bom_4,\
-\bom_2 + \bom_3 + 2\bom_4,\
-\bom_3 + 3\bom_4.
\end{multline*}
The weights of the tensor product $\bcU^{\bom_3} \otimes \bcU^{\bom_3}$ are pairwise sums of the above weights;
potential highest weights of the irreducible summands among these are the sums of the highest weight~$\bom_3$ and another weight.
The only~$\rB_3$-dominant weights among these are the weights
\begin{equation*}
2\bom_3,\ \bom_2 + \bom_4,\ \bom_1 + 2\bom_4,\ 3\bom_4.
\end{equation*}
The corresponding vector bundles $\bcU^{2\bom_3}$, $\bcU^{\bom_2}(1)$, $\bcU^{\bom_1}(2)$, $\cO(3)$
have ranks $35$, $21$, $7$, and~$1$, respectively; they sum up to~$64$, the rank of~$\bcU^{\bom_3} \otimes \bcU^{\bom_3}$, hence
\begin{equation*}
\bcU^{\bom_3} \otimes \bcU^{\bom_3} \cong \bcU^{2\bom_3} \oplus \bcU^{\bom_2}(1) \oplus \bcU^{\bom_1}(2) \oplus \cO(3).
\end{equation*}
The only way to cook up the rank-36 and the rank-28 summands $\Sym^2\bcU^{\bom_3}$ and $\wedge^2\bcU^{\bom_3}$ out of these four
gives the lemma.
\end{proof}

A similar argument proves

\begin{lemma}
\label{lemma:squares-2}
We have an isomorphism
\begin{equation*}
\wedge^2\bcU^{\bom_1} = \bcU^{\bom_2}.
\end{equation*}
\end{lemma}

\end{document}